\documentclass[a4paper,11pt,leqno]{amsart}
\usepackage{amsmath,amssymb,mathrsfs}

\newtheorem{theorem}{Theorem}[section]
\newtheorem{lemma}[theorem]{Lemma}
\newtheorem{proposition}[theorem]{Proposition}
\newtheorem{corollary}[theorem]{Corollary}

\theoremstyle{definition}

\newtheorem{example}[theorem]{Example}

\theoremstyle{remark}
\newtheorem{remark}[theorem]{Remark}

\numberwithin{equation}{section}

\long\def\comment#1{}

\DeclareFontEncoding{OT2}{}{}
\DeclareFontSubstitution{OT2}{cmr}{m}{n}
\DeclareSymbolFont{cyss}{OT2}{wncyss}{m}{n}
\DeclareMathSymbol{\sh}{\mathbin}{cyss}{`x}

\allowdisplaybreaks

\begin{document}
\title[Sum formulas for double polylogarithms]{Sum formulas for double polylogarithms with a shifting parameter and their derivatives}
\author{Kohji Matsumoto}
\address{K. Matsumoto: Graduate School of Mathematics, Nagoya University, Chikusa-ku, Nagoya 464-8602 Japan}
\email{kohjimat@math.nagoya-u.ac.jp}

\author{Hirofumi Tsumura}
\address{{H.\,Tsumura:} Department of Mathematics and Information Sciences, Tokyo Metropolitan University, 1-1, Minami-Ohsawa, Hachioji, Tokyo 192-0397 Japan}
\email{tsumura@tmu.ac.jp}

\begin{abstract}
We prove sum formulas for double polylogarithms of Hurwitz type, that
is, involving a shifting parameter $b$ in the denominator.   These
formulas especially imply well-known sum formulas for double zeta
values, and sum formulas for double $L$-values.   Further,
differentiating in $b$, we obtain a kind of weighted sum formula for
double polylogarithms and double $L$-values.    We also give 
sum formulas for partial double zeta values with some congruence
conditions.    Our proofs of those sum formulas are based on certain
functional relations for double polylogarithms of Hurwitz type.
\end{abstract}
\subjclass[2010]{11M32 (primary), 11M35 (secondary)}
\keywords{Double zeta values, polylogarithms, Hurwitz zeta-functions, Dirichlet $L$-functions}

\maketitle

\baselineskip 15pt

\section{Introduction and the statement of main results} \label{sec-1}

Let ${\Bbb N}$ be the set of positive integers, 
${\Bbb Z}$ the ring of rational integers, ${\Bbb Q}$
the rational number field, ${\Bbb R}$ the real number field, ${\Bbb  C}$ the complex number field, respectively.  

The double zeta values are defined by
\begin{equation}
\zeta_2(k_1,k_2)=\sum_{1\leq m_1<m_2}\frac{1}{m_1^{k_1}m_2^{k_2}}=\sum_{m=1}^\infty \sum_{n=1}^\infty \frac{1}{m^{k_1}(m+n)^{k_2}} \label{e-1-1}
\end{equation}
for $k_1,k_2\in \mathbb{N}$ with $k_2>1$. This was first considered by Euler. Hence this is sometimes called the Euler sum. It is well-known that the notion of \eqref{e-1-1} was extended to the multiple zeta value (MZV):
\begin{equation}
\zeta_r(k_1,k_2,\ldots,k_r)=\sum_{1\leq m_1<m_2<\cdots <m_r} \frac{1}{m_1^{k_1}\cdots m_r^{k_r}} \label{e-1-2}
\end{equation}
for $k_1,k_2,\ldots,k_r\in \mathbb{N}$ with $k_r>1$, defined by Hoffman \cite{Ho} and Zagier \cite{Za}. Moreover a lot of interesting relation formulas among MZVs have been discovered and the structure of the $\mathbb{Q}$-algebra generated by MZVs has been actively studied (see, for example, \cite{BB0,BB2,Gon,Ho3,Kaneko,Tera}). 

One of the famous formulas among MZVs is the sum formula. As for the double zeta values, it can be expressed as 
\begin{equation}
\sum_{j=2}^{l-1}\zeta_2(l-j,j)=\zeta(l) \qquad (l\in \mathbb{N};\, l\geq 3), \label{sumformula}
\end{equation}
where $\zeta(s)$ is the Riemann zeta-function, which was first proved by Euler.  

In recent days, further studies on sum formulas have been actively pursued
by many mathematicians.
In \cite{GKZ}, Gangl, Kaneko and Zagier obtained the following formulas, which
``divide'' \eqref{sumformula} for even $l$ into two parts: 
\begin{equation}
\begin{split}
& \sum_{\nu=1}^{N-1} \zeta_2(2N-2\nu,2\nu)=\frac{3}{4}\zeta(2N);\\
& \sum_{\nu=1}^{N-1} \zeta_2(2N-2\nu-1,2\nu+1)=\frac{1}{4}\zeta(2N)
\end{split}
\label{F-GKZ}
\end{equation}
for $N\in \mathbb{N}$ with $N\geq 2$.
Ohno and Zudilin \cite{O-Zu} showed the weighted sum formula
\begin{equation}
\sum_{\nu=2}^{l-1}2^\nu \zeta_2(l-\nu,\nu)=(l+1)\zeta(l) \quad (l\geq 3), \label{F-OZu}
\end{equation}
and Guo and Xie extended \eqref{F-OZu} to that 
for multiple zeta values in \cite{Guo-Xie}. 
Nakamura \cite{Na-Sh} gave the following analogues of \eqref{F-GKZ} and \eqref{F-OZu}:
\begin{align}
& \sum_{\nu=1}^{N-1} \left( 4^\nu+4^{N-\nu}\right) \zeta_2(2N-2\nu,2\nu)=\left( N+\frac{4}{3}+\frac{2}{3}4^{N-1}\right)\zeta(2N); \label{Naka-1}\\
& \sum_{\nu=2}^{M-2} (2\nu-1)(2M-2\nu-1)\zeta_2(2M-2\nu,2\nu)=\frac{3}{4}(M-3)\zeta(2M) \label{Naka-2}
\end{align}
for $N,M\in \mathbb{N}$ with $N\geq 2$ and $M\geq 4$.

The first purpose of the present paper is to give an extension of \eqref{sumformula} 
to double polylogarithms of Hurwitz type with a shifting parameter $b$ as follows. 
Note that the empty sum is to be 
understood as $0$. 

\begin{theorem} \label{Main-T}
For $k\in \mathbb{N}$, $x\in \mathbb{C}$ with $|x|\leq 1$ and $b\in (0,1]$,
\begin{align}
& \cos (b\pi) \bigg\{ \sum_{m\geq 1 \atop n\geq 0}\frac{x^n}{m(m+n+b)^{k+1}}-\sum_{m\geq 1 \atop n\geq 0}\frac{x^{m+n}}{m(m+n+b)^{k+1}}\label{SF-Hurwitz}\\
& \quad +\sum_{\nu=2}^{k+1} \sum_{m\geq 1 \atop n\geq 0}\frac{x^n}{(n+b)^{k+2-\nu}(m+n+b)^\nu}-\sum_{m\geq 1 \atop n\geq 0}\frac{x^{m+n}}{(n+b)(m+n+b)^{k+1}}\notag\\
& \quad + \sum_{m> b \atop n\geq 0}\frac{x^n}{(m-b)(m+n)^{k+1}}+\sum_{\nu=2}^{k+1}\sum_{m> b \atop n\geq 0}\frac{x^n}{(n+b)^{k+2-\nu}(m+n)^\nu}\bigg\}\notag \\
& -\frac{\sin (b\pi)}{\pi} \bigg\{ \sum_{m\geq 1 \atop n\geq 0}\frac{x^n}{m(m+n+b)^{k+2}}-\sum_{m\geq 1 \atop n\geq 0}\frac{x^{m+n}}{m(m+n+b)^{k+2}}\notag\\
& \quad +\sum_{\nu=1}^{k} \sum_{m\geq 1 \atop n\geq 0}\frac{x^n}{(n+b)^\nu(m+n+b)^{k+3-\nu}}-\sum_{m\geq 1 \atop n\geq 0}\frac{x^{m+n}}{(n+b)(m+n+b)^{k+2}}\notag\\
& \quad -\sum_{m\geq 1 \atop n\geq 0}\frac{x^{m+n}}{(n+b)^2(m+n+b)^{k+1}}-\sum_{m> b \atop n\geq 0}\frac{x^n}{(m-b)^2(m+n)^{k+1}}\notag\\
& \quad-k\sum_{m> b \atop n\geq 0}\frac{x^n}{(m-b)(m+n)^{k+2}}-\sum_{\nu=1}^{k}\nu\sum_{m> b \atop n\geq 0}\frac{x^n}{(n+b)^{k+1-\nu}(m+n)^{\nu+2}}\bigg\}\notag\\
&\ =\pi \sin (b\pi) \sum_{n\geq 0} \frac{x^n}{(n+b)^{k+1}}+2\cos (b\pi) \sum_{n\geq 0} \frac{x^n}{(n+b)^{k+2}}\notag\\
& \quad -\frac{2\sin (b\pi)}{\pi}\sum_{n\geq 0} \frac{x^n}{(n+b)^{k+3}}. \notag
\end{align}
\end{theorem}

The proof of this theorem will be given in Section \ref{sec-2}.
It is to be stressed that the existence of the parameter $b$ is important for
applications (see Theorem \ref{C-1-4}).

In the case $b=1$, multiplying by $x$ the both sides of \eqref{SF-Hurwitz}, we have the following.

\begin{corollary} \label{C-1-2}
For $k \in \mathbb{N}$ and $x\in \mathbb{C}$ with $|x|\leq 1$,
\begin{align}                                                                  
& \sum_{m\geq 1 \atop n\geq 1}\frac{x^n}{m(m+n)^{k+1}}-\sum_{m\geq 1 \atop n\geq 1}\frac{x^{m+n}}{m(m+n)^{k+1}}+\sum_{\nu=2}^{k+1}\sum_{m\geq 1 \atop n\geq 1}\frac{x^n}{n^{k+2-\nu}(m+n)^{\nu}}\label{SF-EMT}\\
& \quad ={\rm Li}(k+2;x),\notag
\end{align}
where ${\rm Li}(s;x)=\sum_{n\geq 1}\, {x^n}{n^{-s}}$.
\end{corollary}

In particular when $x=1$, \eqref{SF-EMT} implies the sum formula \eqref{sumformula}. 
Hence \eqref{SF-Hurwitz} is a generalization of \eqref{sumformula}.
Note that this corollary was already obtained by Essouabri and the authors in 
\cite[Theorem 3]{EMT}.   In this sense, the present paper is a continuation of
\cite{EMT}.

Moreover we give 
certain sum formulas for double $L$-values as follows. For any Dirichlet character $\chi$ of conductor $f$, we denote $\chi^{-1}$ by $\overline{\chi}$ and let $\tau(\overline{\chi})=\sum_{a=1}^{f}\overline{\chi}(a)e^{2\pi i a/f}$ be the Gauss sum. Then, multiplying by $\overline{\chi}(a)$ the both sides of \eqref{SF-EMT} with $x=e^{2\pi i a/f}$, summing up with $a=1,\ldots,f$, and using 
$$\chi(n)=\frac{1}{\tau(\overline{\chi})}\sum_{a=1}^{f}\overline{\chi}(a)e^{2\pi i an/f}\qquad (n\in \mathbb{N})$$
(see \cite[Lemma 4.7]{Wa}), we obtain the following. 

\begin{corollary} \label{C-1-3}
For $k \in \mathbb{N}$,
\begin{align}                                                                  
& \sum_{\nu=2}^{k+1}L_2^{\sh}(k+2-\nu,\nu;\chi,\chi_0)+L_2^{\sh}(1,k+1;\chi_0,\chi)-L_2^{\ast}(1,k+1;\chi_0,\chi) \label{sumf-L}\\
& \quad =L(k+2;\chi),\notag
\end{align}
where $\chi_0$ is the trivial character, $L(s;\chi)=\sum_{n\geq 1}\chi(n)n^{-s}$ is the Dirichlet $L$-function and
\begin{align}                                                                  
& L_2^{\ast}(s_1,s_2;\psi_1,\psi_2)=\sum_{1\leq m_1<m_2}\frac{\psi_1(m_1)\psi_2(m_2)}{m_1^{s_1}m_2^{s_2}}=\sum_{m\geq 1 \atop n\geq 1}\frac{\psi_1(m)\psi_2(m+n)}{m^{s_1}(m+n)^{s_2}}, \label{def-L-ast}\\
& L_2^{\sh}(s_1,s_2;\psi_1,\psi_2)=\sum_{1\leq m_1<m_2}\frac{\psi_1(m_1)\psi_2(m_2-m_1)}{m_1^{s_1}m_2^{s_2}}=\sum_{m\geq 1 \atop n\geq 1}\frac{\psi_1(m)\psi_2(n)}{m^{s_1}(m+n)^{s_2}} \label{def-L-sh}
\end{align}
are double $L$-functions for Dirichlet characters $\psi_1,\psi_2$, which were first studied in \cite{AK04}.
\end{corollary}

Note that recently another new approach to \eqref{sumf-L} has been given by Kawashima, Tanaka and Wakabayashi in \cite{KTW}. In fact, they studied cyclic sum formulas for multiple $L$-values including \eqref{sumf-L}.

In order to prove Theorem \ref{Main-T}, we give functional relations for 
certain double polylogarithms of Hurwitz type (see Theorem \ref{T-2-1}), which is a generalization of our previous result (see \cite[Theorem 2.1]{MT-Siau}). From these functional relations, we can obtain Theorem \ref{Main-T} by using the technique of partial fraction decompositions. This method was already introduced in our previous work (see \cite[Remark 2.4]{MT-Quart}; also \cite[Section 7]{EMT}). 

It should be emphasized that \eqref{SF-Hurwitz} holds as a functional relation among holomorphic functions for $b\in D_\varepsilon(1)$, where we fix a sufficiently small $\varepsilon>0$ and let
\begin{equation}
D_\varepsilon(1)=\{ b\in \mathbb{C}\,|\, |b-1|<\varepsilon\}  \label{def-D-1}
\end{equation}
(see Lemma \ref{L-2-1}). 
Therefore it is possible to differentiate \eqref{SF-Hurwitz} with respect to the shifting parameter $b$ to obtain various new formulas. 
By differentiating \eqref{SF-Hurwitz} in $b$, we obtain, for example, the following theorem. The proof will be given in Section \ref{sec-2-2}.


\begin{theorem} \label{C-1-4}
For $k \in \mathbb{N}$ and $x\in \mathbb{C}$ with $|x|\leq 1$,
\begin{align}                                                                   & \sum_{\nu=2}^{k+1} \nu\,\sum_{m\geq 1 \atop n\geq 1}\frac{x^n}{n^{k+2-\nu}(m+n)^{\nu+1}}+ (k+1)\sum_{m\geq 1 \atop n\geq 1}\frac{x^n}{m(m+n)^{k+2}}\label{AK-poly}\\
& \quad \ +\sum_{m\geq 1 \atop n\geq 1}\frac{x^n}{m^2(m+n)^{k+1}}-\sum_{m\geq 1 \atop n\geq 1}\frac{x^{m+n}}{m^2(m+n)^{k+1}}=\sum_{n\geq 1} \frac{x^n}{n^{k+3}}.\notag
\end{align}
\end{theorem}

Similarly to Corollary \ref{C-1-3}, for a Dirichlet character $\chi$ of conductor $f$, multiplying by $\overline{\chi}(a)$ the both sides of \eqref{AK-poly} with $x=e^{2\pi i a/f}$, and summing up with $a=1,\ldots,f$, we obtain the following.

\begin{corollary} \label{C-3-1}
For $k \in \mathbb{N}$,
\begin{align}                                                                   & \sum_{\nu=2}^{k+1} \nu\,L_2^{\sh}(k+2-\nu,\nu+1;\chi,\chi_0)+(k+1)L_2^{\sh}(1,k+2;\chi_0,\chi)\label{AK-L-1}\\
& \quad +L_2^{\sh}(2,k+1;\chi_0,\chi)-L_2^{\ast}(2,k+1;\chi_0,\chi)=L(k+3;\chi)
.\notag
\end{align}
In particular,
\begin{align}       
\sum_{\nu=2}^{k} \nu\,\zeta_2(k+2-\nu,\nu+1) +2(k+1)\zeta_2(1,k+2)&=\zeta(k+3).  
\label{SF-nu}                                                                   \end{align}
\end{corollary}

Note that \eqref{SF-nu} can also be derived from the result of Arakawa and Kaneko 
\cite[Corollary 11]{AK99} for MZVs (see Remark \ref{rem-AK}). 

The above corollaries imply that Theorem \ref{Main-T} is the source of various
sum formulas among double $L$-values and double polylogarithms. 
In fact, considering the derivatives of higher order, we can obtain other kinds of 
sum formulas 
among double $L$-values and double polylogarithms (see Remark \ref{Rem-3-4}).  
Therefore it is an important problem to generalize Theorem \ref{Main-T} to the
multiple case, aiming to extend the sum formula for MZVs.

Next we give a sum formula for the partial double polylogarithms modulo $N$. 
We insert this result in the present paper, because its proof shares some common 
features with that of Theorem \ref{Main-T}.
We will give its proof in Section \ref{sec-3}, where a shifting analogue will
also be presented (Theorem \ref{T-4-4}).

\begin{theorem} \label{T-4-1}
Let $N$ be any odd positive integer. For $k\in \mathbb{N}$ and $x\in \mathbb{C}$ with $|x|\leq 1$,
\begin{align}                                                                          
& \sum_{\nu=1}^{k} \bigg\{ \sum_{m,n\geq 1 \atop {m \equiv n\,(\text{mod\ }N)}}
\frac{x^n}{n^\nu(m+n)^{k+2-\nu}}\bigg\}+\sum_{m,n\geq 1 \atop {m \equiv n\,
(\text{mod\ }N)}}
\frac{x^n}{m(m+n)^{k+1}} \label{4-1}\\   
& \qquad +\sum_{\nu=1}^{k} \bigg\{ \sum_{m,n\geq 1 \atop {m \equiv -2n\,
(\text{mod\ }N)}} \frac{x^n}{n^\nu(m+n)^{k+2-\nu}}\bigg\}+ \sum_{m,n\geq 1 \atop {m \equiv -2n\,
(\text{mod\ }N)}} \frac{x^n}{m(m+n)^{k+1}}\notag\\
& \qquad - \sum_{m,n\geq 1 \atop {m \equiv -2n\,
(\text{mod\ }N)}} \frac{x^{m+n}}{m(m+n)^{k+1}}- \sum_{m,n\geq 1 \atop {m \equiv -2n\,
(\text{mod\ }N)}} \frac{x^{m+n}}{n(m+n)^{k+1}}\notag\\
& \qquad =\frac{2}{N^{k+2}}{\rm Li}(k+2;x^N)+\frac{\pi}{N}\sum_{m\geq 1\atop N\nmid m} 
\frac{x^m}{\sin(2m\pi/N)m^{k+1}}.\notag   
\end{align}
\end{theorem}

In particular when $N=1$, we have \eqref{SF-EMT}, because 
the second term on the right-hand side of \eqref{4-1} vanishes. In the case $N=3$ and $x=1$, we 
see that
$m \equiv -2n\,(\text{mod\ }3)$ implies $m \equiv n\,(\text{mod\ }3)$. Also we see 
that $\chi_3(m)=(2/\sqrt{3})\sin(2m\pi/3)$, where $\chi_3$ is the quadratic character of 
conductor $3$. Hence we have the following.

\begin{corollary} \label{C-4-2}
For $k\in \mathbb{N}$,
\begin{align}                                                                           
& \sum_{\nu=1}^{k} \bigg\{ \sum_{m,n\geq 1 \atop {m \equiv n\,(\text{mod\ }3)}}
\frac{1}{n^\nu(m+n)^{k+2-\nu}}\bigg\}=\frac{1}{3^{k+2}}\zeta(k+2)+\frac{\pi}
{3\sqrt{3}}L(k+1;\chi_3). \label{4-2}   
\end{align}
\end{corollary}

This corollary suggests that 
it is also interesting to generalize Theorem \ref{T-4-1} to the multiple case. 


\bigskip

{\bf Acknowledgements}
The authors would like to express their sincere gratitude to Professor Yasuo Ohno 
for pointing out that \eqref{SF-nu} in Corollary \ref{C-3-1} is a consequence of the work of 
Arakawa and Kaneko \cite{AK99}.

\bigskip

\section{Functional relations and the proof of Theorem \ref{Main-T}} \label{sec-2}

Let 
\begin{equation}
\phi (s)=\sum_{m\geq 1} \frac{(-1)^m}{m^{s}}=(2^{1-s}-1)\zeta(s).
\label{def-phi}
\end{equation}
From \cite[Eq.\,(4.32)]{KMT-CJ}, we see that
\begin{align}
& \lim_{M \to \infty}\sum_{m=1}^{M} \frac{(-1)^{m}\sin(m\theta)}{m}=\phi(0)\theta=-\frac{\theta}{2} \label{e-2-1}
\end{align}
for $\theta \in (-\pi,\pi) \subset \mathbb{R}$. 

We prove the following functional relations for double polylogarithms of Hurwitz type. This result is an extension of our previous result given in \cite[Theorem 2.1]{MT-Siau} which corresponds to the case $b=1$.

\begin{theorem} \label{T-2-1}
For $b \in \mathbb{R}$ with $0<b\leq 1$, $x \in \mathbb{C}$ with $|x| \leq 1$ and $s \in \mathbb{C}$ with $\Re s\geq 1$, 
\begin{align}
& \cos(b\pi)\bigg\{\sum_{m \geq 1 \atop n\geq 0}\frac{x^n}{m(n+b)^s(m+n+b)}-\sum_{m \geq 1 \atop n\geq 0}\frac{x^{m+n}}{m(n+b)(m+n+b)^s} \label{2-2}\\
& \qquad \qquad+\sum_{m >b \atop n\geq 0}\frac{x^n}{(m-b)(n+b)^s(m+n)}\bigg\}\notag\\
& -\frac{\sin(b \pi)}{\pi}\bigg\{ \sum_{m \geq 1 \atop n\geq 0}\frac{x^n}{m(n+b)^s(m+n+b)^2}-\sum_{m \geq 1 \atop n\geq 0}\frac{x^{m+n}}{m(n+b)^2(m+n+b)^s}\notag\\
& \qquad \qquad+\sum_{m >b \atop n\geq 0}\frac{x^n}{(m-b)^2(n+b)^s(m+n)}\bigg\}\notag\\
& =\pi \sin(b\pi)\sum_{n\geq 0}\frac{x^n}{(n+b)^{s+1}}-2\cos(b\pi)\sum_{n\geq 0}\frac{x^n}{(n+b)^{s+2}} +\frac{\sin(b\pi)}{\pi}\sum_{n\geq 0}\frac{x^n}{(n+b)^{s+3}}.\notag
\end{align}
\end{theorem}

\begin{proof}
First we assume $\Re s>1$. Then, by (\ref{e-2-1}), 
\begin{align}
& 2i\left(\lim_{M \to \infty}\sum_{m=1}^{M} \frac{(-1)^{m}\sin(m\theta)}{m}+\frac{\theta}{2}\right) \sum_{n\geq 0} \frac{(-1)^n x^n e^{i(n+b)\theta}}{(n+b)^s}=0 \label{e-2-6} 
\end{align}
holds for $\theta \in (-\pi,\pi)$. 
It follows from \cite[$\S$ 3.35]{WW} that the former term in the parentheses on the left-hand side of (\ref{e-2-6}) is uniformly convergent on any compact subset in $(-\pi,\pi)$ though it is not absolutely convergent. Further the latter part on the left-hand side is absolutely and uniformly convergent for $\theta\in (-\pi,\pi)$ because $\Re s>1$. 
Changing the order of summation, we see that 
\begin{align}
& \lim_{M \to \infty}\sum_{m= 1}^{M}\sum_{n\geq 0} \frac{(-1)^{m+n}x^n e^{i(m+n+b)\theta}}{m(n+b)^s} -\lim_{M \to \infty}\sum_{m= 1}^{M}\sum_{n\geq 0 \atop m\not=n+b} \frac{(-1)^{m+n}x^n e^{i(-m+n+b)\theta}}{m(n+b)^s} \label{FE-poly} \\
& \notag \qquad +{i\theta}\sum_{n\geq 0} \frac{(-1)^{n}x^n e^{i(n+b)\theta}}{(n+b)^s}
=
\begin{cases}
-\sum_{n\geq 0}\frac{x^n}{(n+1)^{s+1}} & (b=1)\\
0 & (b\not=1), 
\end{cases}
\end{align}
where the left-hand side is uniformly convergent on any compact subset in 
$(-\pi,\pi)$ though it is not absolutely convergent.   Note that the change of 
the order of summation can be justified (see \cite[$\S$\
\,4]{MT-Quart} or \cite[Section 7]{EMT}).

Denote the left-hand side of \eqref{FE-poly} by $J(\theta)$ and the right-hand side by $C_0$, respectively. Multiply by $(i/2\pi)\theta$ the both sides and consider the integration 
$$\frac{i}{2\pi}\int_{-\pi}^{\pi}\theta J(\theta)d\theta= \frac{iC_0}{2\pi}\int_{-\pi}^{\pi} \, \theta\, d\theta=0.$$
This gives that
\begin{align*}
& \cos(b \pi) \sum_{m \geq 1 \atop n\geq 0}\frac{x^n}{m(n+b)^s(m+n+b)}
-\frac{\sin(b \pi)}{\pi}\sum_{m \geq 1 \atop n\geq 0}\frac{x^n}{m(n+b)^s(m+n+b)^2}\\
& -\cos(b \pi) \sum_{m \geq 1 \atop {n\geq 0 \atop m\not=n+b}}\frac{x^n}{m(n+b)^s(m+n+b)}+\frac{\sin(b \pi)}{\pi}\sum_{m \geq 1 \atop {n\geq 0 \atop m\not=n+b}}\frac{x^n}{m(n+b)^s(m+n+b)^2}\\
& -\pi \sin(b\pi)\sum_{n\geq 0}\frac{x^n}{(n+b)^{s+1}}-2\cos(b\pi)\sum_{n\geq 0}\frac{x^n}{(n+b)^{s+2}} +\frac{\sin(b\pi)}{\pi}\sum_{n\geq 0}\frac{x^n}{(n+b)^{s+3}}=0.
\end{align*}
Note that the double sums on the left-hand side are determined independently from the order of summation because these are absolutely convergent. 
We can rewrite the third and the fourth members on the left-hand side to 
\begin{align*}
& -\cos(b \pi)\left\{ - \sum_{l>b \atop n\geq 0}\frac{x^n}{(n+l)(n+b)^s(l-b)}
+ \sum_{m\geq 1 \atop j\geq 0}\frac{x^{m+j}}{m(m+j+b)^s(j+b)}\right\}\\
& +\frac{\sin(b \pi)}{\pi}\left\{ \sum_{l>b \atop n\geq 0}\frac{x^n}{(n+l)(n+b)^s(l-b)^2}
+ \sum_{m\geq 1 \atop j\geq 0}\frac{x^{m+j}}{m(m+j+b)^s(j+b)^2}\right\},
\end{align*}
by putting $l=m-n\,(>b)$ and $j=-m+n\,(\geq 0)$ corresponding to $m>n+b$ and $m<n+b$, respectively. Hence we see that \eqref{2-2} holds for $\Re s>1$, further holds for $\Re s\geq 1$, because the both sides of \eqref{2-2} are absolutely convergent for $\Re s\geq 1$. Thus we have the assertion.
\end{proof}

Now we give the proof of Theorem \ref{Main-T} as follows.

\begin{proof}[Proof of Theorem \ref{Main-T}]

Letting $s=k \in \mathbb{N}$ in \eqref{2-2} and using the partial fraction decomposition 
$$\frac{1}{XY}=\left(\frac{1}{X}+\frac{1}{Y}\right)\frac{1}{X+Y} $$
repeatedly, we can rewrite the former three terms in the parentheses on the left-hand side of (\ref{2-2}) to 
\begin{align}
& \sum_{\nu=2}^{k+1}\sum_{m \geq 1 \atop n\geq 0}\frac{x^n}{(n+b)^{k+2-\nu}(m+n+b)^\nu}+\sum_{m \geq 1 \atop n\geq 0}\frac{x^n}{m(m+n+b)^{k+1}}\label{2-3}\\
& -\sum_{m \geq 1 \atop n\geq 0}\frac{x^n}{(n+b)(m+n+b)^{k+1}}-\sum_{m \geq 1 \atop n\geq 0}\frac{x^n}{m(m+n+b)^{k+1}} \notag\\
& +\sum_{\nu=2}^{k+1}\sum_{m >b \atop n\geq 0}\frac{x^n}{(n+b)^{k+2-\nu}(m+n)^\nu}+\sum_{m >b \atop n\geq 0}\frac{x^n}{(m-b)(m+n)^{k+1}}.\notag
\end{align}
Also we can similarly decompose the latter four terms in the parentheses on the left-hand side of (\ref{2-2}). In particular, the last term can be decomposed as
\begin{align}
& \sum_{m >b \atop n\geq 0}\frac{x^n}{(m-b)^2(n+b)^k(m+n)} \label{2-4}\\
& =\sum_{m >b \atop n\geq 0}\frac{x^n}{(m-b)^2(n+b)^{k-1}(m+n)^2}+\sum_{m >b \atop n\geq 0}\frac{x^n}{(m-b)(n+b)^k(m+n)^2}.\notag
\end{align}
In order to repeat this decomposition, we put 
$$S_{\mu}=\sum_{m >b \atop n\geq 0}\frac{x^n}{(m-b)^2(n+b)^{k+1-\mu}(m+n)^{\mu}}\qquad (1\leq \mu\leq k+1).$$
Then by the same consideration as above, we have
\begin{align*}
S_{\mu}-S_{\mu+1}& =\sum_{m >b \atop n\geq 0}\frac{x^n}{(m-b)(n+b)^{k+1-\mu}(m+n)^{\mu+1}}\\
&=\sum_{\nu=\mu+2}^{k+2}\sum_{m >b \atop n\geq 0}\frac{x^n}{(n+b)^{k+3-\nu}(m+n)^{\nu}}+\sum_{m >b \atop n\geq 0}\frac{x^n}{(m-b)(m+n)^{k+2}} 
\end{align*}
for $1\leq \mu\leq k$. Summing up these relations for $\mu=1,2,\ldots,k$, we obtain 
\begin{align*}
S_1=& S_{k+1}+\sum_{\mu=1}^{k}\sum_{\nu=\mu+2}^{k+2}\sum_{m >b \atop n\geq 0}\frac{x^n}{(n+b)^{k+3-\nu}(m+n)^{\nu}}+k\sum_{m >b \atop n\geq 0}\frac{x^n}{(m-b)(m+n)^{k+2}}\\
=& \sum_{m >b \atop n\geq 0}\frac{x^n}{(m-b)^2(m+n)^{k+1}}+\sum_{\nu=3}^{k+2}(\nu-2)\sum_{m >b \atop n\geq 0}\frac{x^n}{(n+b)^{k+3-\nu}(m+n)^{\nu}}\\
& \qquad +k\sum_{m >b \atop n\geq 0}\frac{x^n}{(m-b)(m+n)^{k+2}},
\end{align*}
which completes the decomposition \eqref{2-4}. Combining \eqref{2-3} and these considerations, we obtain \eqref{SF-Hurwitz}.
\end{proof}

%

\section{The proof of Theorem \ref{C-1-4}} \label{sec-2-2}

First we fix $k\in \mathbb{N}$, $x \in \mathbb{C}$ with $|x|\geq 1$, assume $0<b<1$, and differentiate the both sides of \eqref{SF-Hurwitz} with respect to the shifting parameter $b$. It should be noted that on the left-hand side of \eqref{SF-Hurwitz}, the second term in the former parentheses and the third and the fourth terms in the latter parentheses are not continuous at $b=1$, because
\begin{align*}
& \sum_{m> b \atop n\geq 0}\frac{x^n}{(m-b)(m+n)^{k+2}} = 
\frac{1}{x(1-b)}{\rm Li}(k+2;x)+\sum_{m \geq 2 \atop n\geq 0}\frac{x^n}{(m-b)(m+n)^{k+2}},\\
& \sum_{m> b \atop n\geq 0}\frac{x^n}{(m-b)^2(m+n)^{k+1}}=\frac{1}{x(1-b)^2}{\rm Li}(k+1;x)+\sum_{m \geq 2 \atop n\geq 0}\frac{x^n}{(m-b)^2(m+n)^{k+1}},\\
& \sum_{m> b \atop n\geq 0}\frac{x^n}{(m-b)(m+n)^{k+2}}=\frac{1}{x(1-b)}{\rm Li}(k+2;x)+\sum_{m\geq 2 \atop n\geq 0}\frac{x^n}{(m-b)(m+n)^{k+2}}.
\end{align*}
Hence we pick up these non-continuous terms from the left-hand side of \eqref{SF-Hurwitz}. Namely we define $g(b)=g_k(b;x)$ by 
\begin{equation}
g(b)=\frac{\cos(b\pi)}{x(1-b)}{\rm Li}(k+1;x)+\frac{\sin(b\pi)}{\pi x(1-b)^2}{\rm Li}(k+1;x)+\frac{k\sin(b\pi)}{\pi x(1-b)}{\rm Li}(k+2;x). \label{def-g}
\end{equation}

\begin{lemma} \label{L-2-1} Let $k\in \mathbb{N}$ and $x\in \mathbb{C}$ with $|x|\leq 1$. Then $g(b)=g_k(b;x)$ is holomorphic for $b\in D_\varepsilon(1)$ (defined in \eqref{def-D-1}) and satisfies
\begin{equation}
g(1)=\frac{k}{x}{\rm Li}(k+2;x),\ g'(1)=-\frac{\pi^2}{3x}{\rm Li}(k+1;x),\ g''(1)=-\frac{k\pi^2}{3x}{\rm Li}(k+2;x). \label{g-deri}
\end{equation}
Furthermore \eqref{SF-Hurwitz} holds as a functional relation among holomorphic functions for $b\in D_\varepsilon(1)$. 
\end{lemma}

\begin{proof}
First study the part $g(b)$ of the left-hand side of \eqref{SF-Hurwitz}. By considering the Taylor expansion of each term in \eqref{def-g} around $b=1$, we can easily see that $g(b)$ is holomorphic for $b\in D_\varepsilon(1)$ and satisfies \eqref{g-deri}. The remainder parts on the left-hand side of \eqref{SF-Hurwitz} are absolutely convergent for $b\in D_\varepsilon(1)$. On the other hand, we can see that the right-hand side of \eqref{SF-Hurwitz} is absolutely convergent for $b\in D_\varepsilon(1)$. Thus we have the assertion.
\end{proof}

It follows from this lemma that the left-hand side of \eqref{SF-Hurwitz} can be written as 
\begin{equation}
\cos(b\pi)H_1(b)-\frac{\sin(b\pi)}{\pi}H_2(b)+g(b), \label{left-hand}
\end{equation}
by picking $g(b)$ up from the terms in each parenthesis on the left-hand side of \eqref{SF-Hurwitz} and denoting each remainder part by $H_1(b)$ and $H_2(b)$, respectively. Then $H_1(b)$ and $H_2(b)$ are holomorphic for $b \in D_\varepsilon(1)$ and can be differentiated term-wisely because of the absolute convergency. Hence we differentiate \eqref{left-hand} in $b$, multiply by $x$, replace $n$ by $n+1$ and let $b \to 1$. Then, by \eqref{g-deri}, we have
\begin{align}
x & \lim_{b \to 1}\left\{-H'_1(b)+H_2(b)+g'(b)\right\}  \label{equ-1}\\
 =&(k+1) \sum_{m,n \geq 1}\frac{x^n}{m(m+n)^{k+2}}-(k+1) \sum_{m,n \geq 1}\frac{x^{m+n}}{m(m+n)^{k+2}}  \notag\\
& +\sum_{\nu=2}^{k+1}\left\{(k+2-\nu)\sum_{m,n \geq 1}\frac{x^n}{n^{k+3-\nu}(m+n)^\nu}+\nu\sum_{m,n \geq 1}\frac{x^n}{n^{k+2-\nu}(m+n)^{\nu+1}}\right\} \notag\\
& - \sum_{m,n \geq 1}\frac{x^{m+n}}{n^2(m+n)^{k+1}}-(k+1) \sum_{m,n \geq 1}\frac{x^{m+n}}{n(m+n)^{k+2}}- \sum_{m,n \geq 1}\frac{x^{n}}{m^2(m+n)^{k+1}}\notag\\
& +\sum_{\nu=2}^{k+1}(k+2-\nu)\sum_{m,n \geq 1}\frac{x^n}{n^{k+3-\nu}(m+n-1)^\nu}\notag\\
& + \sum_{m,n \geq 1}\frac{x^{n}}{m(m+n)^{k+2}}-\sum_{m,n \geq 1}\frac{x^{m+n}}{m(m+n)^{k+2}}+\sum_{\nu=1}^{k}\sum_{m,n \geq 1}\frac{x^n}{n^{\nu}(m+n)^{k+3-\nu}}\notag\\
& - \sum_{m,n \geq 1}\frac{x^{m+n}}{n(m+n)^{k+2}}- \sum_{m,n \geq 1}\frac{x^{m+n}}{n^2(m+n)^{k+1}}\notag\\
& - \sum_{m,n \geq 1}\frac{x^{n}}{m^2(m+n)^{k+1}}- k\sum_{m,n \geq 1}\frac{x^{n}}{m(m+n)^{k+2}}\notag\\
& -\sum_{\nu=3}^{k+2}(\nu-2)\sum_{m,n \geq 1}\frac{x^n}{n^{k+3-\nu}(m+n-1)^{\nu}}\notag\\
& -\frac{\pi^2}{3}\sum_{n\geq 1}\frac{x^n}{n^{k+1}}.\notag
\end{align}
The fourth and the eighth lines on the right-hand side are 
\begin{align*}
& \sum_{\nu=2}^{k+1}(k+2-\nu)\sum_{m,n \geq 1}\frac{x^n}{n^{k+3-\nu}(m+n)^\nu}+\sum_{\nu=2}^{k+1}(k+2-\nu)\sum_{n \geq 1}\frac{x^n}{n^{k+3}},\\
& -\sum_{\mu=2}^{k+1}(\mu-1)\sum_{m,n \geq 1}\frac{x^n}{n^{k+2-\mu}(m+n)^{\mu+1}}-\sum_{\mu=2}^{k+1}(\mu-1)\sum_{n \geq 1}\frac{x^n}{n^{k+3}},
\end{align*}
respectively, and
their second members are cancelled with each other. 
We apply \eqref{SF-EMT} to the fifth line on the right-hand side of \eqref{equ-1}. Then it can be rewritten to 
$$\sum_{n\geq 1}\frac{x^n}{n^{k+3}}-\sum_{m,n \geq 1}\frac{x^n}{n^{k+1}(m+n)^2}.$$
Hence the right-hand side of \eqref{equ-1} is equal to
\begin{align}
& \sum_{m,n \geq 1}\frac{x^n}{m(m+n)^{k+2}} \label{equ-2}\\
& -(2k+3) \sum_{m,n \geq 1}\frac{x^{m+n}}{m(m+n)^{k+2}} -2 \sum_{m,n \geq 1}\frac{x^{n}}{m^2(m+n)^{k+1}}-2 \sum_{m,n \geq 1}\frac{x^{m+n}}{m^2(m+n)^{k+1}}\notag\\
& +2\sum_{\nu=2}^{k+1}(k+2-\nu)\sum_{m,n \geq 1}\frac{x^n}{n^{k+3-\nu}(m+n)^\nu}\notag\\
& +\sum_{\nu=2}^{k+1}\sum_{m,n \geq 1}\frac{x^n}{n^{k+2-\nu}(m+n)^{\nu+1}}\notag\\
& - \sum_{m,n \geq 1}\frac{x^{n}}{n^{k+1}(m+n)^{2}}+\sum_{n\geq 1}\frac{x^n}{n^{k+3}} -2\zeta(2)\sum_{n\geq 1}\frac{x^n}{n^{k+1}}.\notag
\end{align}
By applying \eqref{SF-EMT} to the first and the fourth lines in \eqref{equ-2}, they can be rewritten to 
$$ \sum_{m,n \geq 1}\frac{x^{m+n}}{m(m+n)^{k+2}} -\sum_{m,n \geq 1}\frac{x^n}{n^{k+1}(m+n)^{2}}+\sum_{n \geq 1}\frac{x^n}{n^{k+3}}.$$
Furthermore we can see that
\begin{align}
\zeta(2)\sum_{l\geq 1}\frac{x^l}{l^{k+1}} & = \bigg\{ \sum_{1\leq m<l}+\sum_{1\leq l<m}+\sum_{1\leq m=l}\bigg\} \frac{x^l}{m^2 l^{k+1}} \label{equ-3}\\
& =\sum_{m,n \geq 1}\frac{x^{m+n}}{m^2(m+n)^{k+1}}+\sum_{m,n \geq 1}\frac{x^{m}}{m^{k+1}(m+n)^{2}}+\sum_{n \geq 1}\frac{x^n}{n^{k+3}}.\notag
\end{align}
Therefore \eqref{equ-2} can be rewritten to 
\begin{align}
& -(2k+2) \sum_{m,n \geq 1}\frac{x^{m+n}}{m(m+n)^{k+2}} -2 \sum_{m,n \geq 1}\frac{x^{n}}{m^2(m+n)^{k+1}} \label{equ-4}\\
& +2\sum_{\nu=2}^{k+1}(k+2-\nu)\sum_{m,n \geq 1}\frac{x^n}{n^{k+3-\nu}(m+n)^\nu}+4\sum_{n\geq 1}\frac{x^n}{n^{k+3}} -4\zeta(2)\sum_{n\geq 1}\frac{x^n}{n^{k+1}},\notag
\end{align}
which is equal to the value $x \lim_{b \to 1}\left\{-H'_1(b)+H_2(b)+g'(b)\right\}$ in \eqref{equ-1}. 
On the other hand, we differentiate the right-hand side of \eqref{SF-Hurwitz} in $b$, multiply by $x$ and let $b \to 1$. Then we have 
\begin{equation}
-\pi^2 \sum_{n\geq 1}\frac{x^n}{n^{k+1}}+2(k+3)\sum_{n\geq 1}\frac{x^n}{n^{k+3}}. \label{equ-5}
\end{equation}
Combining \eqref{equ-4} and \eqref{equ-5}, we obtain the following proposition. 
\begin{proposition} \label{P-3-1}
For $k\in \mathbb{N}$, 
\begin{align}
& \sum_{\nu=2}^{k+1}(k+2-\nu)\sum_{m,n \geq 1}\frac{x^n}{n^{k+3-\nu}(m+n)^\nu} \label{Prop-3-1}\\
& \quad -\sum_{m,n \geq 1}\frac{x^n}{m^2(m+n)^{k+1}}-(k+1)\sum_{m,n \geq 1}\frac{x^{m+n}}{m(m+n)^{k+2}}\notag\\
& \quad =(k+1) \sum_{n\geq 1}\frac{x^n}{n^{k+3}}-\zeta(2)\sum_{n\geq 1}\frac{x^n}{n^{k+1}}\notag.
\end{align}
\end{proposition}

Based on these considerations, we now prove Theorem \ref{C-1-4}. 

\begin{proof}[Proof of Theorem \ref{C-1-4}] 
By applying \eqref{SF-EMT}, the first member on the left-hand side of \eqref{Prop-3-1} can be rewritten to 
\begin{align*}
& (k+1)\sum_{\nu=2}^{k+1}\sum_{m,n \geq 1}\frac{x^n}{n^{k+3-\nu}(m+n)^\nu}-\sum_{\mu=1}^{k}\mu \sum_{m,n \geq 1}\frac{x^n}{m^2(m+n)^{\mu+1}}\\
& \ =(k+1) \bigg\{ \sum_{n\geq 1}\frac{x^n}{n^{k+3}}-\sum_{m,n \geq 1}\frac{x^n}{m(m+n)^{k+2}}\notag \\
& \quad +\sum_{m,n \geq 1}\frac{x^{m+n}}{m(m+n)^{k+2}}-\sum_{m,n \geq 1}\frac{x^n}{n(m+n)^{k+2}}\bigg\}\notag \\
& \quad -\sum_{\mu=2}^{k+1}\,\mu \sum_{m,n \geq 1}\frac{x^n}{n^{k+2-\mu}(m+n)^{\mu+1}}-\sum_{m,n \geq 1}\frac{x^n}{m^{k+1}(m+n)^{2}}\notag\\
& \quad +(k+1)\sum_{m,n \geq 1}\frac{x^{n}}{n(m+n)^{k+2}}.\notag
\end{align*}
Hence \eqref{Prop-3-1} can be rewritten to 
\begin{align*}
& (k+1)\sum_{n\geq 1}\frac{x^n}{n^{k+3}} -(k+1)\sum_{m,n \geq 1}\frac{x^n}{m(m+n)^{k+2}} \\
& -\sum_{\mu=2}^{k+1}\,\mu \sum_{m,n \geq 1}\frac{x^n}{n^{k+2-\mu}(m+n)^{\mu+1}} -\sum_{m,n \geq 1}\frac{x^n}{n^{k+1}(m+n)^{2}}-\sum_{m,n \geq 1}\frac{x^n}{m^{2}(m+n)^{k+1}}\notag\\
& \ =(k+1) \sum_{n\geq 1}\frac{x^n}{n^{k+3}}-\zeta(2)\sum_{n \geq 1}\frac{x^{n}}{n^{k+1}}.\notag
\end{align*}
Therefore, by \eqref{equ-3}, we complete the proof of Theorem \ref{C-1-4}.
\end{proof}

\begin{remark} \label{rem-AK}
As mentioned in Section \ref{sec-1}, \eqref{SF-nu} 
can be derived from the case $(r,k)=(1,2)$ of a formula of Arakawa and Kaneko 
\cite[Corollary 11]{AK99}.
Their formula is of a more general feature and was proved by considering analytic 
properties of 
\begin{equation}
\xi(k_1,\ldots,k_r;s)=\frac{1}{\Gamma(s)}\int_{0}^\infty \frac{t^{s-1}}{e^t-1}\cdot {{\rm Li}}_{r}({k_1,\ldots,k_r};1-e^{-t}) dt, \label{AK-xi}
\end{equation}
where
$$\textrm{Li}_{r}(k_1,\ldots,k_r;z)=\sum_{0<m_1<\cdots<m_r}\frac{z^{m_r}}{m_1^{k_1}\cdots m_r^{k_r}}\qquad (k_1,\ldots,k_r\in \mathbb{N}).$$
In addition, it is to be noted that the duality for MZVs plays an important role in 
the proof of their formula. 
Therefore, at present, it is unclear whether the method of Arakawa and Kaneko can be 
used to prove \eqref{AK-poly} and \eqref{AK-L-1}. 
This is one of the reasons why we propose, in Section \ref{sec-1}, the problem of
generalizing Theorem \ref{Main-T} to the multiple case; because this would lead,
without using the duality, 
to multiple generalizations of \eqref{AK-poly} and \eqref{AK-L-1}, that would imply 
an $L$-analogue of the Arakawa-Kaneko formula.
\end{remark}

\begin{remark}\label{Rem-3-4}
If we repeatedly differentiate the both sides of \eqref{SF-Hurwitz} with respect to $b$ and consider its limit as $b\to 1$, then we can obtain other formulas analogous to \eqref{sumformula}. For example, considering the derivative of the second order, we can obtain the following formula, while we omit its proof for the purpose of saving space.
\begin{align}
& \sum_{\nu=3}^{k}\frac{\nu(\nu-1)}{2}\sum_{m\geq 1\atop n\geq 0}\frac{x^n}{(n+1)^{k+3-\nu}(m+n+1)^{\nu+1}} \label{higher}\\
& \ +\frac{(k+2)(k+3)}{6}\left\{ \sum_{m\geq 1\atop n\geq 0}\frac{x^{m+n}-x^n}{m(m+n+1)^{k+3}}-\sum_{m\geq 1\atop n\geq 0}\frac{x^{m+n}-x^n}{(n+1)(m+n+1)^{k+3}}\right\} \notag\\
& \ -(k+1)\left\{ \sum_{m\geq 1\atop n\geq 0}\frac{x^{m+n}-x^n}{m^2(m+n+1)^{k+2}}-\sum_{m\geq 1\atop n\geq 0}\frac{x^{m+n}-x^n}{(n+1)^2(m+n+1)^{k+2}}\right\} \notag\\
& \ -\sum_{m\geq 1\atop n\geq 0}\frac{x^{m+n}-x^n}{m^3(m+n+1)^{k+1}}\notag\\
& \ +\frac{k(k+2)}{3}\sum_{m\geq 1\atop n\geq 0}\frac{x^n}{(n+1)(m+n+1)^{k+3}}+\frac{(k+2)(2k+3)}{3}\sum_{m\geq 1\atop n\geq 0}\frac{x^n}{m(m+n+1)^{k+3}}\notag\\
& \ +\frac{(k+1)(k+2)}{2}\sum_{m\geq 1\atop n\geq 0}\frac{x^n}{(n+1)^2(m+n+1)^{k+2}}=\sum_{n\geq 0}\frac{x^n}{(n+1)^{k+4}}\qquad (k\geq 3).\notag
\end{align}
Also we can obtain an $L$-analogue similarly to Corollaries \ref{C-1-3} and 
\ref{C-3-1}. Note that, letting $x=1$ in \eqref{higher}, we obtain 
\begin{align}
& \sum_{\nu=3}^{k}\frac{\nu(\nu-1)}{2}\zeta_2(k+3-\nu,\nu+1) +\frac{(k+1)(k+2)}{2}\zeta_2(2,k+2)\label{AK-F2}\\
& \qquad  +(k+1)(k+2)\zeta_2(1,k+3)=\zeta(k+4)\qquad (k\geq 3), \notag 
\end{align}
which can be also derived from \cite[Corollary 11]{AK99} in the case $(r,k)=(2,2)$. 
\end{remark}

\bigskip

%
\section{Proof of Theorem \ref{T-4-1} and the shifting analogue}\label{sec-3}
%

In this section, we prove Theorem \ref{T-4-1}.
For this purpose, we first prove the following functional relations.

%
%

\begin{proposition}\label{P-4-3}
Let $N$ be any odd positive integer. Then
\begin{align}
& \sum_{m,n\geq 1 \atop {m \equiv n\,(\text{mod\ }N)}}\frac{x^n}{mn^s(m+n)}+\sum_{m,n\geq 1 \atop {m \equiv -2n\,(\text{mod\ }N)}}\frac{x^n}{mn^s(m+n)} -\sum_{m,n\geq 1 \atop {m \equiv -2n\,(\text{mod\ }N)}}\frac{x^{m+n}}{mn(m+n)^s}\label{4-3}\\
& \qquad =\frac{2}{N^{s+2}}{\rm Li}(s+2;x^N)+\frac{\pi}{N}\sum_{m=1\atop N\nmid m}^\infty \frac{x^m}{\sin(2m\pi/N)m^{s+1}}\notag
\end{align}
holds for $s\in \mathbb{C}$ with $\Re s\geq 1$.
\end{proposition}

\begin{proof}
We fix $s\in \mathbb{C}$ with $\Re s>1$. 
In \eqref{FE-poly}, we let $b=1$, multiply by $(-x)$, and replace $n+1$ by $n$. Then we have
\begin{align}
& \lim_{M\to \infty}\sum_{m=1}^{M}\sum_{n\geq 1} \frac{(-1)^{m+n}x^n e^{i(m+n)\theta}}{mn^s} -\lim_{M\to \infty}\sum_{m=1}^{M}\sum_{n\geq 1 \atop m\not=n} \frac{(-1)^{m+n}x^n e^{i(-m+n)\theta}}{mn^s} \label{FE-poly-2} \\
& \qquad +{i\theta}\sum_{n\geq 1} \frac{(-1)^{n}x^n e^{in\theta}}{n^s}
=\sum_{n\geq 1}\frac{x^n}{n^{s+1}},\notag
\end{align}
which is uniformly convergent on any compact subset of $(-\pi,\pi)$ though it is not absolutely convergent as we mentioned in Section \ref{sec-2}. Hence we can consider the indefinite integral of \eqref{FE-poly-2} on $(-\pi,\pi)$, namely
\begin{align}
& \frac{1}{i}\sum_{m\geq 1}\sum_{n\geq 1} \frac{(-1)^{m+n}x^n e^{i(m+n)\theta}}{mn^s(m+n)} -\frac{1}{i}\sum_{m\geq 1}\sum_{n\geq 1 \atop m\not=n} \frac{(-1)^{m+n}x^n e^{i(-m+n)\theta}}{mn^s(-m+n)} \label{FE-poly-3} \\
& \qquad +{\theta}\sum_{n\geq 1} \frac{(-1)^{n}x^n e^{in\theta}}{n^{s+1}}-\frac{1}{i}\sum_{n\geq 1} \frac{(-1)^{n}x^n e^{in\theta}}{n^{s+2}}=C_0^*+\theta \sum_{n\geq 1}\frac{x^n}{n^{s+1}}\ \ (\theta \in (-\pi,\pi))\notag
\end{align}
with a certain constant $C_0^*$. 
Now we consider the definite integral of \eqref{FE-poly-3} from $-\pi$ to $\pi$. Then, by noting
$$\int_{-\pi}^{\pi} e^{in\theta} d\theta=0\qquad (n\in \mathbb{Z}\setminus \{0\}),$$
we have
$$C_0^*=\frac{1}{i}\sum_{n\geq 1}\frac{x^n}{n^{s+2}}.$$
On the both sides of \eqref{FE-poly-3}, we set $\eta=e^{2\pi i/N}$ for an odd positive integer $N$, $\theta=2\pi a/N-\pi$, and replace $x$ by $x\eta^{-2a}$ for $a=0,1,\ldots,N-1$. Then, multiplying by $i$, we have 
\begin{align}
& \sum_{m\geq 1}\sum_{n\geq 1} \frac{x^n\eta^{a(m-n)}}{mn^s(m+n)} -\sum_{m\geq 1}\sum_{n\geq 1 \atop m\not=n} \frac{x^n\eta^{-a(m+n)}}{mn^s(-m+n)} +\frac{i\pi(2a-N)}{N}\sum_{n\geq 1} \frac{x^n\eta^{-an}}{n^{s+1}}\label{FE-poly-4} \\
& \quad -\sum_{n\geq 1} \frac{x^n\eta^{-an}}{n^{s+2}}=\sum_{n\geq 1} \frac{x^n\eta^{-2an}}{n^{s+2}}+\frac{i\pi(2a-N)}{N} \sum_{n\geq 1}\frac{x^n\eta^{-2an}}{n^{s+1}}.\notag
\end{align}
In general, for any primitive $N$th root $\xi$ of unity, by considering 
$$\frac{1-X^N}{1-X}=1+X+\cdots+X^{N-1}$$
and its derivation, we immediately obtain that 
\begin{equation*}
\sum_{a=0}^{N-1}\xi^{-an}=
\begin{cases}
N & (N \mid n) \\
0 & (N \nmid n),
\end{cases}
\end{equation*}
and 
\begin{equation*}
\sum_{a=0}^{N-1}a\xi^{-an}=
\begin{cases}
\frac{N(N-1)}{2} & (N \mid n) \\
\frac{N}{\xi^{-n}-1} & (N \nmid n).
\end{cases}
\end{equation*}
Hence we have
\begin{equation*}
\sum_{a=0}^{N-1}\frac{2a-N}{N}\xi^{-an}=
\begin{cases}
-1 & (N \mid n) \\
\frac{2}{\xi^{-n}-1} & (N \nmid n).
\end{cases}
\end{equation*}
Therefore, summing up \eqref{FE-poly-4} for $a=0,1,\ldots,N-1$ and using the above results with $\xi=\eta,\eta^2$ (because $N$ is odd), we obtain \eqref{4-3} for $s\in \mathbb{C}$ with $\Re s>1$. Further we can see that \eqref{4-3} holds for $\Re s\geq 1$. Thus we have the assertion.
\end{proof}

\begin{proof}[Proof of Theorem \ref{T-4-1}]
Put $s=k\in \mathbb{N}$ on the both sides of \eqref{4-3}. Then, by the same method as in the proof of Theorem \ref{Main-T}, namely, by the partial fraction decomposition, we see that the first term on the left-hand side of \eqref{4-3} is equal to 
$$\sum_{\nu=1}^{k}\sum_{m,n\geq 1 \atop {m \equiv n\,(\text{mod\ }N)}}\frac{x^n}{n^\nu(m+n)^{k+2-\nu}}+\sum_{m,n\geq 1 \atop {m \equiv n\,(\text{mod\ }N)}}\frac{x^n}{m(m+n)^{k+1}}.$$
Similarly we can decompose the second and the third terms on the left-hand side of \eqref{4-3}. Therefore we consequently obtain \eqref{4-1}.
\end{proof}
%

To prove Proposition \ref{P-4-3}, we put $b=1$ in \eqref{FE-poly}.
It seems difficult to show an analogue of Proposition \ref{P-4-3} for general $b$.
However, at least in the case $b=1/2$, we can deduce a formula 
analogous to Proposition \ref{P-4-3}, which gives a shifting analogue of
Theorem \ref{T-4-1}.
We conclude the present paper with the statement of those results.

\begin{proposition}\label{P-4-3-2}
Let $N$ be any odd positive integer. Then
\begin{align*}
& \sum_{m,n\geq 0 \atop {m \equiv n\,(\text{mod\ }N)}}\frac{x^n}{(m+1/2)(n+1/2)^s(m+n+1)}\\
& \qquad+\sum_{m,n\geq 0 \atop {m \equiv -2n-2\,(\text{mod\ }N)}}\frac{x^n}{(m+1)(n+1/2)^s(m+n+3/2)} \\
& \qquad -\sum_{m,n\geq 0 \atop {m \equiv -2n-2\,(\text{mod\ }N)}}\frac{x^{m+n+1}}{(m+1)(n+1/2)(m+n+3/2)^s}\\
& \qquad =\frac{x^{(N-1)/2}}{N^{s+2}}\sum_{n\geq 0} \frac{x^{Nn}}{(n+1/2)^{s+2}}+\frac{\pi}{N}\sum_{m\geq 0\atop N\nmid (2m+1)} \frac{x^m}{\sin((2m+1)\pi/N)(m+1/2)^{s+1}}\notag
\end{align*}
holds for $s\in \mathbb{C}$ with $\Re s\geq 1$.
\end{proposition}

The proof of this proposition is quite similar to that of Proposition \ref{P-4-3},
and we omit it.
From this fact, we can prove the following result similarly to Theorem \ref{T-4-1}.

\begin{theorem} \label{T-4-4}
Let $N$ be any odd positive integer. For $k\in \mathbb{N}$ and $x\in \mathbb{C}$ with $|x|\leq 1$,
\begin{align}
& \sum_{\nu=1}^{k} \bigg\{ \sum_{m,n\geq 0 \atop {m \equiv n\,(\text{mod\ }N)}}
\frac{x^n}{(n+1/2)^\nu(m+n+1)^{k+2-\nu}}\bigg\} \label{E-4-2-1}\\
& \ +\sum_{m,n\geq 0 \atop {m \equiv n\,
(\text{mod\ }N)}}
\frac{x^n}{(m+1/2)(m+n+1)^{k+1}}\notag\\
& \ +\sum_{\nu=1}^{k} \bigg\{ \sum_{m,n\geq 0 \atop {m \equiv -2n-2\,
(\text{mod\ }N)}} \frac{x^n}{(n+1/2)^\nu(m+n+3/2)^{k+2-\nu}}\bigg\}\notag \\
& \ + \sum_{m,n\geq 0 \atop {m \equiv -2n-2\,
(\text{mod\ }N)}} \frac{x^n}{(m+1)(m+n+3/2)^{k+1}}\notag\\
& \ - \sum_{m,n\geq 0 \atop {m \equiv -2n-2\,
(\text{mod\ }N)}} \frac{x^{m+n+1}}{(m+1)(m+n+3/2)^{k+1}}\notag\\
& \ - \sum_{m,n\geq 0 \atop {m \equiv -2n-2\,
(\text{mod\ }N)}} \frac{x^{m+n+1}}{(n+1/2)(m+n+3/2)^{k+1}}\notag\\
& \qquad =\frac{x^{(N-1)/2}}{N^{k+2}}\sum_{n\geq 0} \frac{x^{Nn}}{(n+1/2)^{k+2}}+\frac{\pi}{N}\sum_{m\geq 0\atop N\nmid (2m+1)} \frac{x^m}{\sin((2m+1)\pi/N)(m+1/2)^{k+1}}.\notag
\end{align}
\end{theorem}

\begin{example} Set $x=1$ and $k=1$ in \eqref{E-4-2-1}.
Then the last four sums on the left-hand side cancel with each other, and so we have
\begin{align*}
& \sum_{m,n\geq 0 \atop {m \equiv n\,
(\text{mod\ }N)}} \frac{1}{(n+1/2)(m+n+1)^{2}}\\
& \qquad =\frac{1}{2N^{3}}\sum_{n\geq 0} \frac{1}{(n+1/2)^{3}}+\frac{\pi}{2N}\sum_{m\geq 0\atop N\nmid (2m+1)} \frac{1}{\sin((2m+1)\pi/N)(m+1/2)^{2}}.\notag
\end{align*}
In particular when $N=1$, we have
\begin{align*}
\sum_{m,n\geq 0} \frac{1}{(n+1/2)(m+n+1)^{2}}& =\frac{1}{2}\sum_{n\geq 0} \frac{1}{(n+1/2)^{3}}=\frac{7}{2}\zeta(3).
\end{align*}
When $N=3$, since
$\sin((2m+1)\pi/3)=(\sqrt{3}/2)\chi_3(2m+1)$, we have 
\begin{align*}
\sum_{m,n\geq 0 \atop {m \equiv n\,
(\text{mod\ }3)}} \frac{1}{(n+1/2)(m+n+1)^{2}}& =\frac{7}{54}\zeta(3)
+\frac{5\pi}{3\sqrt{3}}L(2,\chi_3).\notag
\end{align*}
\end{example}

\end{document}